\documentclass[11pt, a4paper]{article}

\usepackage[utf8]{inputenc} 
\usepackage[T1]{fontenc}

\usepackage{amsmath, amssymb, amsthm}

\usepackage[margin=2.5cm]{geometry}   
\usepackage{hyperref}                 
\usepackage{graphicx}              
\usepackage{xcolor}                 

\usepackage{subcaption}
\usepackage{amssymb,amsmath,amsfonts,amsxtra,amsthm,stmaryrd}
\usepackage{mathrsfs}
\usepackage{frcursive}
\usepackage{color}
\usepackage{tkz-tab}
\usepackage{float}
\usepackage{tikz}
\usepackage{siunitx}
\usepackage[european, straightvoltages, RPvoltages]{circuitikz}
\usepackage{authblk}

\DeclareMathOperator{\Mult}{Mult}

\newcommand{\Nn}{\mathbb{N}}
\newcommand{\Qq}{\mathbb{Q}}

\newcommand{\Cc}{\mathbb{C}}
\newcommand{\Pp}[1]{\mathbb{C}\{\!\{#1\}\!\}}
\newcommand{\CC}[1]{(\Cc^*)^#1}
\newcommand{\M}[1]{\Mult_{#1}}
\newcommand{\Mun}{\Mult_{\textbf{1}}}

\newcommand{\W}{\text{W}}

\usepackage{aliascnt}

\newcommand{\mynewtheorem}[2]{
  \newaliascnt{#1}{dummy}
  \newtheorem{#1}[#1]{#2}
  \aliascntresetthe{#1}
  \expandafter\def\csname #1autorefname\endcsname{#2}
}
 
\theoremstyle{plain}
  \mynewtheorem{theorem}{Theorem}
  \mynewtheorem{proposition}{Proposition}
  \mynewtheorem{corollaire}{Corollary}
  \mynewtheorem{lemme}{Lemma}

\theoremstyle{definition}
  \mynewtheorem{definition}{Definition}
  \mynewtheorem{exemple}{Example}
\theoremstyle{remark}
  \mynewtheorem{remark}{Remark}
  \mynewtheorem{question}{Question}
\newtheorem*{thm:associativity}{Theorem \ref{thm:associativity}}

\title{Intersection multiplicity of a sparse curve and a trinomial curve}
\author{Jean Mousseau}
\affil{Laboratoire de Mathématiques, UMR 5127 CNRS, Universit{\'e} Savoie Mont Blanc} 
\date{}

\begin{document}

\maketitle

\begin{abstract}

    Let $f \in \Cc[x,y]$ be a trinomial and let $g \in \Cc[x,y]$ be a sparse polynomial with $N\geq 1$ monomials.
     We address the problem of bounding the local intersection multiplicity of the bivariate polynomial system $f = g = 0$ at an isolated
      solution $q \in \CC{2}$ only in terms of $N$. Our main result establishes that this local multiplicity is at most $\frac{(N-1)(N+2)}{2}$, independently 
      of the degrees of $f$ and $g$. Furthermore, we show that this bound is sharp for $N=3$. 

\end{abstract}

\section{Introduction}

In classical algebraic geometry, the complexity of a polynomial system is typically measured by the degrees of the equations, as epitomized by Bézout's
 theorem. This theorem bounds the sum of the intersection multiplicities of all isolated solutions by the product of the degrees. However, fewnomial theory 
 shifts this perspective, seeking bounds that depend solely on the number of non-zero terms (monomials) appearing in the polynomials.

By Descartes' rule of signs \cite{descartes2022geometrie}, the number of positive roots of a univariate polynomial with
 $t$ terms does not exceed $t-1$. Khovanskii 
\cite{khovanskii1991fewnomials} generalized this type of bound to multivariate polynomials: a system of $n$ polynomials in $n$ variables involving a total of $t$ distinct
 monomials has at most 
\begin{equation*}    
    2^{\binom{t-1}{2}} (n + 1)^{t-1}
\end{equation*}
non-degenerate positive solutions. Setting $m = t - n - 1$ as the \emph{codimension} of the system, this bound was later dramatically improved by Bihan
 and Sottile \cite{bihan2007new} to
\begin{equation*}
    \frac{e^2 + 3}{4} 2^{\binom{m}{2}} n^{m}.
\end{equation*}

While bounding the number of real roots has been extensively studied since Khovanskii's foundational work \cite{khovanskii1991fewnomials}, controlling the 
intersection multiplicity of complex solutions in terms of sparse data remains largely open. The first result in this direction was an analogue of Descartes' 
rule of signs due to Hajós \cite{hajos1953solution}: the multiplicity of any nonzero root of a univariate polynomial with $t$ monomials is at most $t - 1$.
 For multivariate systems of $n$ equations in $n$ variables with a total of $t$ distinct monomials, Gabrielov \cite{gabrielov1995multiplicities} established 
 the upper bound 
\begin{equation*}
    2^{\binom{t}{2}}(n+1)^t
\end{equation*}
on the multiplicity of any isolated solution in $\CC{n}$. However, his bound is of order $2^{\mathcal{O}(t^2)}$, leaving the existence of a purely polynomial 
bound in $t$ (for fixed $n$) as an open question.

A striking parallel emerges when comparing the known polynomial bounds (for fixed $n$) for the number of positive solutions with those for intersection 
multiplicities in the algebraic torus $\CC{n}$:

     \textit{Dimension one ($n=1$):} Descartes' rule of signs gives the sharp bound $t-1$ for positive roots, while Hajós' lemma \cite{hajos1953solution}
     yields the exact same bound $t-1$ for root multiplicity.

     \textit{Codimension zero ($m=0$):} such systems can be reduced to binomial systems and have at most $1$ non-degenerate positive solution, and similarly
     intersection multiplicity at most $1$ in the torus \cite{bihan2025sparsesystemshighlocal}.

     \textit{Codimension one ($m=1$):} Bihan \cite{bihan2007polynomial} proved that a system of $n$ equations in $n$ variables with $t = n + 2$ monomials
     has at most $n + 1$ non-degenerate positive solutions. Using Gale duality, Bihan, Dickenstein, and Forsgård \cite{bihan2025sparsesystemshighlocal} established the exact 
     counterpart for intersection multiplicities, showing that the multiplicity of an isolated torus solution is also bounded by $n + 1$ (moreover, for systems
      that are non-degenerate, they obtained the general polynomial bound $(t-1)^n$).

It is therefore natural to ask whether there is a systematic connection between upper bounds on the number of positive real solutions and upper bounds on the 
intersection multiplicity of a sparse polynomial system. This is precisely \textbf{Question 1.1} posed by Bihan, Dickenstein, and Forsgård in \cite{bihan2025sparsesystemshighlocal}.

For both problems, the first fundamental open frontier arises in dimension two and codimension two ($n=2, m=2$), namely for systems of two 
bivariate polynomials whose combined support consists of at most $t=5$ monomials. Finding a sharp upper bound on either the number of positive solutions 
or the intersection multiplicity for general systems of this type remains an open problem. A natural and fundamental subcase consists in considering a system 
in two variables formed by a trinomial and a polynomial with $N = 3$ monomials. Indeed, for two trinomials, 
up to dividing each equation by one of its monomials, we may assume without loss of generality that both contain a nonzero constant term; their combined support
 then involves at most $5$ monomials, making it a genuine subcase of the dimension $2$ and codimension $2$ problem.

For the number of isolated positive solutions of a bivariate system consisting of a trinomial and an $N$-nomial, Li, Rojas, and Wang \cite{Li_Rojas_Wang_2003}
 established the upper bound $2^N - 2$. More recently, this bound was significantly improved to a cubic polynomial in $N$ by El-Hilany, and Tavenas
  \cite{elhilany2020improved} via Wronskians and \emph{dessins d'enfants}. In the special case of two trinomials ($N=3$), the result of Li, Rojas, and Wang
   shows that such a system has at most $5$ isolated positive solutions. This bound of $5$ is known to be sharp, as previously
   demonstrated by Haas \cite{haas2002simple} with the explicit system
\begin{equation}\label{eq:system de haas}
    10x^{106}+11y^{53}-11y = 10y^{106}+11x^{53}-11x = 0,
\end{equation}
which possesses exactly five positive solutions.

In this paper, we establish the counterpart to these results for intersection multiplicities, obtaining a bound that is \emph{quadratic} in $N$. We consider a
 system $f(x, y) = g(x, y) = 0$, where $f, g \in \Cc[x, y]$ are polynomials with complex coefficients such that $f$ is a trinomial and $g$ has $N > 1$ monomials.
 Our main result is as follows:

\begin{thm:associativity}
   Let $f\in \Cc[x,y]$ be a trinomial and $g\in \Cc[x,y]$ be a polynomial with $N\geq 2$ monomials.
   If $q\in \CC{2}$ is an isolated solution of $f=g=0$, then 
   
    \begin{align*}
        \M{q}(f,g) \leq \frac{(N-1)(N+2)}{2}.
    \end{align*}

    Furthermore, for $N=3$, this bound is sharp.
\end{thm:associativity}


For $N=3$, \autoref{them:main} yields a maximal intersection multiplicity of $5$ for a system of two trinomials, matching the optimal bound of Li, Rojas, 
and Wang for positive solutions. Moreover, we show that this upper bound of $5$ is sharp, as achieved at $q = (1,1)$ by the system
\begin{equation}\label{eqopt5}
    2x^9+y^9-3 = x^{16}+2xy^6-3y = 0. 
\end{equation}
This bound for two trinomials provides a strong affirmative insight into \textbf{Question 1.1}.

In the 1970s, Anatoli Kouchnirenko conjectured that a system of $n$ polynomial equations in $n$ variables with $m_1, \dots, m_n$ monomials, 
respectively, has at most $\prod_{i=1}^n (m_i - 1)$ strictly positive solutions (a conjecture refuted for $m_1=m_2=3$ by Haas via the system \eqref{eq:system de haas}).
 The analogous question for intersection multiplicities in the algebraic torus had remained open until now. System \eqref{eqopt5} demonstrates that 
 $\prod_{i=1}^n (m_i - 1)$ fails to be a valid upper bound for the intersection multiplicity as well.
 


\section*{Acknowledgement}

The author would like to thank his advisor Fr\'ed\'eric Bihan for its guidance.

\section{Some basic concepts}

\subsection{Multiplicity and invariants}

We denote the formal power series ring by $\Cc[[x_1,\dots , x_n]]=\Cc[[\bar{x}]]$.

  \begin{definition}[Multiplicity]\label{def:mult}
      Let $f_1,\dots,f_n \in \Cc[[\bar{x}-q]]$ and $q\in \Cc^n$.
    
    The multiplicity at $q$ of the system $f_1=\dots=f_n=0$ is the dimension over $\Cc$ of the quotient of
     $ \Cc[[\bar{x}-q]]$ by the ideal $(f_1,\dots , f_n)$, that is:
      
        $$
            \M q (f_1,\dots , f_n)= \dim_{\Cc} \left( \frac{ \Cc[[\bar{x}-q]]}{(f_1,\dots , f_n)}\right).
        $$
    \end{definition}
    When $n=2$, we recall some properties of the multiplicity of intersection of two polynomials.
  Multiplicity is invariant under translation:
\begin{align}
    \M {(0,0)} (f_1,f_2) = \M {(q_1,q_2)} (f_1(x-q_1,y-q_2),f_2(x-q_1,y-q_2)) \label{eq:cgmtline}
\end{align}
where $(q_1,q_2)\in \Cc^2 $. Moreover, if $q_1,q_2\not = 0$,
\begin{align}
     \M {(1,1)} (f_1,f_2) = \M {(q_1,q_2)} (f_1(x/q_1,y/q_2),f_2(x/q_1,y/q_2)). \label{eq:cgmttore}
\end{align}

Note that the change of coordinate \eqref{eq:cgmttore} preserves the monomials whereas \eqref{eq:cgmtline} does not.
We also recall that multiplicity is invariant under monomial changes of coordinates.

\begin{theorem}{(Monomial change of coordinates.)}\label{chgmmonomial}
    Let $$\phi_A: x=(x_1,\dots,x_n) \mapsto (x^{a_1},\dots,x^{a_n})$$ where $A=\begin{pmatrix}
        | &  & | \\
        a_1 &\cdots & a_n\\
        | & & |
    \end{pmatrix} \in M_n(\mathbb{Q})$.
    If $\det(A)\neq 0$, then for any $p\in \CC{n}$ and any $q \in \phi_A^{-1}(p)$,
    $$
        \M p (f_1,\dots,f_n) = \M {q} (\phi_A^* f_1,\dots ,\phi_A^* f_n)
    $$
    where $\phi^*_A f(x) = f(\phi_A (x))$ denotes the pullback.
\end{theorem}

\begin{proof}
    First, suppose $p = q = (1, \dots ,1) = \mathbf{1}$. Up to transposition, we have 
    $$
        \text{Jac}(\phi_A)(1,\dots,1) = A.
    $$
    Since $\det(A)\neq 0$, by the holomorphic inverse function theorem, $\phi_A$ is
     a local biholomorphism in a neighborhood of $\mathbf{1}$. Therefore, the pullback $\phi_A^*$
     induces a $\mathbb{C}$-algebra automorphism of the formal power series ring $\Cc[[\bar{x}-\mathbf{1}]]$.
    Since the image of the ideal $(f_1,\dots,f_n)$ under $\phi_A^{*}$ is the ideal $(\phi_A^{*}f_1,\dots,\phi_A^{*}f_n)$, $\phi_A^*$ 
    induces an isomorphism between the quotient $\mathbb{C}$-vector spaces:
    $$
        \Cc[[\bar{x}-\mathbf{1}]]/(f_1,\dots,f_n) \simeq \Cc[[\bar{x}-\mathbf{1}]]/(\phi_A^*f_1,\dots,\phi_A^*f_n).
    $$
    Taking dimensions over $\mathbb{C}$ shows that
    $$
        \Mun (f_1,\dots,f_n) = \Mun (\phi_A^*f_1,\dots,\phi_A^*f_n).
    $$
    For an arbitrary $p \in (\mathbb{C}^*)^n$ and $q \in \phi_A^{-1}(p)$, denoting component-wise multiplication
     and division in $(\mathbb{C}^*)^n$ by $px$ and $x/q$, we have:
    \begin{align*}
        \M p (f_1,\dots ,f_n) &= \Mun (f_1(px),\dots ,f_n(px))\\
        &= \Mun (f_1(p\phi_A(x)),\dots ,f_n(p\phi_A(x))) \\
        &= \M q \left(f_1\left(p\phi_A\left(\frac{x}{q}\right)\right),\dots ,f_n\left(p\phi_A\left(\frac{x}{q}\right)\right)\right)\\
        &= \M q \left(f_1\left(p\phi_A(q^{-1})\phi_A(x)\right),\dots ,f_n\left(p\phi_A(q^{-1})\phi_A(x)\right)\right)\\
        &= \M q (\phi_A^* f_1,\dots ,\phi_A^* f_n),
    \end{align*}
    where we used the multiplicative property $\phi_A(x/q) = \phi_A(q)^{-1}\phi_A(x) = p^{-1}\phi_A(x)$.
\end{proof}

\subsection{Puiseux series}\label{sec:Seriedepuiseux}
We briefly introduce the field of Puiseux series (see~\cite{puiseux} for more details).
\begin{definition}\label{defseriepuise}
    Let $q\in \Cc$.
    A Puiseux series at $q$ is a formal series of the form
    \begin{align*}
        S(x) & = \sum_{i=1}^{+\infty} a_i (x-q)^{\frac{\ell_i}{d}},
    \end{align*}
   where the coefficients $a_i$ are complex numbers, the exponents $(\ell_i)_{i\in  \Nn^*}$ si a strictly increasing sequence of integer numbers and $d$ 
   an nonzero integer number.
   We will denote by $\Pp{x-q}$ the field of Puiseux series at $q$.
\end{definition}

The set of Puiseux series at $q$ forms a field with ordinary addition and multiplication. This field is algebraically closed (see \cite{walker1978algebraic},
 Chapter IV, 3.2).
For Puiseux series at $q$
\begin{align*}
    S(x) & = \sum_{i=1}^{+\infty} a_i (x-q)^{\frac{\ell_i}{d}},
\end{align*}
we will use the formal derivative defined as follows (see \cite{Koiran_2020}):

\begin{align*}
   \frac{ \partial S}{ \partial x}(x):= \sum_{i=1}^{+\infty} a_i\frac{\ell_i}{d} (x-q)^{\frac{\ell_i}{d}-1}
\end{align*}
which we will sometimes denote $S'(x)$. In particular, the formal derivative of a Puiseux series at $q$ is still a Puiseux series at $q$.

The usual rules for differentiating sums and products hold:
\begin{align*}
     \frac{ \partial (S+T)}{ \partial x}(x) &=  \frac{ \partial S}{ \partial x}(x)+ \frac{ \partial T}{ \partial x}(x)\\
     \frac{ \partial (ST)}{ \partial x}(x) &= \frac{ \partial S}{ \partial x}(x)T(x)+\frac{ \partial T}{ \partial x}(x)S(x)
\end{align*}
for $S$ and $T$ in $\Pp {x-q}$.

We define a multiplicity (order) of a Puiseux series as follows.

\begin{definition}
Let $S\in \Pp{x-q}$ be as in \autoref{defseriepuise}.
The multiplicity of $S$ at $q$ is
\begin{align*}
     \M{q} (S(x)) = \frac{\min_{ i\in \Nn^*} ( \ell_i \; | \: a_i \not =0)}{d}.
\end{align*}
    We use the convention $\M{q} (0)=+\infty $.
\end{definition}

This multiplicity satisfies the usual properties of a valuation:

\begin{align}
    \M{q} (S(x)+T(x))&\geq \min (\M{q} (S(x)),\M{q} (T(x)))\label{eq:inesommevaluation}\\
    \M{q} (S(x)T(x)) &= \M{q} (S(x))+\M{q} (T(x)) \label{eq:valsom}
\end{align}
for $S$ and $T$ in  $\Pp x$. Moreover,
\begin{align}
    \M{q} \left(\frac{\partial S}{\partial x} (x)\right) \geq \M{q} \left(S(x)\right) -1. \label{eq:valderival}
\end{align}
with equality if and only if $S$ has non zero valuation.

\subsection{Wronskian and toric operator}
We now introduce the main technical tool. 
Bounding the valuation of a sum of series from above is difficult, wereas the valuation of product is simply the sum of the valuation. 
One way to go from a sum to a product is to use the Wronskian. We will also use the toric operator
    
    \begin{align*}
        \Theta_P = P\frac{\partial}{\partial x}
    \end{align*}
    where $P\in \Cc [x]$.
The toric operator satisfies the same inequality as \eqref{eq:valderival}.
\begin{proposition}\label{prop:inegalité derivé toric}
    Let $f\in \Pp{x}$, $P\in \Cc[x]$, and $q\in \Cc$ such that $P(q)\not=0$. We have
\begin{align}
    \M{q} \left(\Theta_P (f) \right) \geq \M{q} \left(f\right) -1.\label{eq:inegalité derivé toric}
\end{align}
with equality if and only if $\M{q}(f)\not = 0$.

\end{proposition}

\begin{proof}
    If $\M{q}(P)=0$, then 
    \begin{align*}
        \M{q} (\Theta_P(f)) &  = \M{q}\left(P\frac{\partial f}{\partial x}\right)\\
         & = \M{q}(P)+\M{q}\left(\frac{\partial f}{\partial x}\right)\\
         & \geq \M{q}(f) -1,
    \end{align*}
using \eqref{eq:valsom} and \eqref{eq:valderival}.
The equality case is directly given by \eqref{eq:valderival}.

\end{proof}

The toric Wronskian $W_P$ is defined by
  \begin{align*}
    W _P(f_1,\dots, f_k) = \det \left(\Theta_P ^{(i-1)} (f_j)\right)_{(i,j)\in \llbracket 1,k\rrbracket ^2},
  \end{align*}
  where $f_1,\dots , f_k$ are ($k-1$)-times differentiable functions.
   The following lemma allows us to extend to the toric Wronskian several results on the usual Wronskian.

\begin{lemme}\label{lemme : formule lian wronskian}
   Let $P\in \Cc[x]$. Then
    $$
         W _P(f_1,\dots, f_k)=P^{\frac{k(k-1)}{2}} \W (f_1,\dots, f_k)
    $$
 where $f_1,\dots , f_k$ are ($k-1$)-times differentiable functions, and $\W$ is the classical Wronskian.
\end{lemme}
    \begin{proof}
        An easy induction shows that
    \begin{align}\label{eq:egalitétoric}
         \Theta_P ^ {(k)}(f) = \sum_{i=0}^{k} P_i f^{(i)},
    \end{align}
    where the $P_i$ are $\mathcal{C}^{\infty}$ functions and $P_k = P^k$. We deduce by row operations that

    \begin{align*}
     \left|\begin{matrix}
            f_1 & \cdots & f_k \\
             \Theta_P(f_1)  &\cdots  &    \Theta_P(f_k)\\
             \vdots & & \vdots \\
              \Theta_P ^{(k-1)}(f_1) & \cdots  &  \Theta_P ^{(k-1)}(f_k)
        \end{matrix}\right| = 
         \left|\begin{matrix}
            f_1 & \cdots & f_k \\
             Pf_1^{(1)}  &\cdots  &    Pf_k^{(1)}\\
             \vdots & & \vdots \\
              P ^{k-1}f_1^{(k-1)} & \cdots  &  P ^{k-1}f_k^{(k-1)}
        \end{matrix}\right|    
    \end{align*} 
    hence the result.
    \end{proof}

This gives a toric analogue of a classical result on the Wronskian.

\begin{proposition}\label{prop : dep leineaire et Wronskian}
   If $f_1,\dots , f_k$ are analytic functions on a connected open set, and $P\not \equiv0$, then the family $(f_i)$ is linearly dependent
    if and only if $W_P(f_1,\dots , f_k)\equiv 0$.
\end{proposition}

\autoref{lemme : formule lian wronskian} implies formulas \eqref{eq:Wronskian formule 1} and \eqref{eq:Wronskian formule 2} below.
\begin{proposition}\label{formuleWronskian}
   Let $P\in \Cc[x]$, $g,f_1,\dots,f_k$ be real-valued functions $k-1$ differentiable on an interval $I$. Then, on $I$,
we have
    \begin{align}
        W_P(gf_1,\dots , gf_k) = (g)^k W_P(f_1,\dots , f_k).\label{eq:Wronskian formule 1}
    \end{align}
  Furthermore, if $f_1$ does not vanish on $I$, then
    \begin{align}
        W_P(f_1,\dots , f_k) = (f_1)^k W_P \left(\Theta_P\left(\frac{f_2}{f_1}\right),\dots , \Theta_P\left(\frac{f_k}{f_1}\right)\right).\label{eq:Wronskian formule 2}    
    \end{align}
    
\end{proposition}

The following proposition is adapted from \cite[Proposition 2.6]{Koiran_2020}.
\begin{proposition}\label{inegaliteWronskian}
   Let $P\in \Cc[x]$ and $q\in \Cc$ such that $P(q)\not=0$.
   Let $S_1(x),\dots , S_n(x)$ be Puiseux series with nonnegative multiplicity at $q$. Then, we have the inequality
    \begin{align}
        \M{q} (S_1(x)+\cdots + S_n(x))\leq \frac{n(n-1)}{2} + \M{q} (W_P(S_1(x),\dots , S_n(x))).\label{eq:bornewronsk 1}
    \end{align}

    If moreover $\M{q} (\Theta_P^{(i-1)}(S_j))\geq 0 $ for all $i,j\in \llbracket 1,n \rrbracket$, then
    \begin{align}
        \M{q} (S_1(x)+\cdots + S_n(x))\leq n-1 + \M{q} (W_P(S_1(x),\dots , S_n(x))).\label{eq:bornewronsk 2}
    \end{align} 
\end{proposition}

\begin{proof}
   We first prove \eqref{eq:bornewronsk 1}. Set
    $$
        S=S_1+\dots+S_n
    $$
    and $M=\M{q} (S)$. 
    
  By linearity of $\Theta_P$, we have

    \begin{align}\label{eq:syst xronsk}
        \begin{pmatrix}
            S_1 & \cdots & S_n \\
            \vdots & \ddots & \vdots \\
            \Theta_P^{(n-1)} (S_1) &  \cdots & \Theta_P^{(n-1)} ( S_n)
        \end{pmatrix}
        \begin{pmatrix}
            1 \\
            \vdots \\
            1
        \end{pmatrix}
        =
        \begin{pmatrix}
             S \\
            \vdots \\
            \Theta_P^{(n-1)} (S) 
        \end{pmatrix}.
    \end{align}

The determinant of the square matrix in \eqref{eq:syst xronsk} is exactly the toric Wronskian. Now, if the toric Wronskian is identically zero,
    \eqref{eq:bornewronsk 1} and
\eqref{eq:bornewronsk 2} hold trivially, since the right-hand sides are infinite. Otherwise, we define $\mathcal{W}_P(x)=W_P(S_1,\dots , S_n)$
 as the determinant of the linear transformation \eqref{eq:syst xronsk}
and we obtain
      \begin{align}\label{eq:syst xronsk121}
        \mathcal{W}_P(x)
        \begin{pmatrix}
            1 \\
            \vdots \\
            1
        \end{pmatrix}
        =
        \text{com}\left(\begin{pmatrix}
            S_1 & \cdots & S_n \\
            \vdots & \ddots & \vdots \\
            \Theta_P^{(n-1)} (S_1) &  \cdots & \Theta_P^{(n-1)} ( S_n)
        \end{pmatrix}\right)^T\begin{pmatrix}
             S \\
            \vdots \\
            \Theta_P^{(n-1)} (S) 
        \end{pmatrix}
    \end{align}
    where $\text{com}(M)^T$ is the adjugate of $M$.
 The first row of \eqref{eq:syst xronsk121} gives
    \begin{align*}
        \mathcal{W}_P(x) = \sum_{i=1}^{n} \Theta_P^{(i-1)} (S)M_i,
    \end{align*}
where $M_i$ is the cofactor of the entry $(i,1)$ of the matrix in \eqref{eq:syst xronsk}.
   By \eqref{eq:inegalité derivé toric}, for $i\in \Nn$,
    \begin{align}\label{eq:min sur dtS}
        \M{q}(\Theta_P^{(i)}(S)) \geq M-i,
    \end{align}
   and, for $i\in \Nn$ and  $j\in \llbracket 1,n \rrbracket$,
    \begin{align}\label{eq:mindupisj}
        \M{q}(\Theta_P^{(i)}(S_j)) \geq -i
    \end{align}
    since $\M{q}(S_j) \geq 0$ by assumption.
   This gives a lower bound for the multiplicity at $q$ of each $M_i$:
    \begin{align}
        \begin{split}
        \M{q} (M_i) &= \M{q} \left(\sum_{\substack{\sigma \in \mathfrak{S}_n \\ \sigma(1) = i}}\varepsilon(\sigma)\prod_{j=2}^{n}\Theta_P^{(\sigma(j)-1)}(S_{j})\right)\\
        & \geq \min_{\substack{\sigma \in \mathfrak{S}_n \\ \sigma(1) = i}} \left(\sum _{j=2}^{n} \M{q}\big(\Theta_P^{(\sigma(j)-1)}(S_{j})\big)\right)\\
        & \geq i-1-\sum_{j=1}^n (j-1) = i-1-\frac{n(n-1)}{2},
        \end{split}\label{eq:min sur mi}
    \end{align}
    where the last inequality follows from \eqref{eq:mindupisj}.
    Combining \eqref{eq:min sur mi} and \eqref{eq:min sur dtS}, we get, for all $i \in \llbracket 1,n \rrbracket$
    \begin{align}\label{jklm}
        \M{q} (  \Theta_P^{(i-1)} (S)M_i) & \geq i-1-\frac{n(n-1)}{2}+M-(i-1)=M-\frac{n(n-1)}{2}.
    \end{align}

    By \eqref{eq:inesommevaluation} and \eqref{jklm}, we get
    \begin{align*}
        \M{q} (\mathcal{W}_P(x)) & = \M{q} \left(\sum_{i=1}^{n} \Theta_P^{(i-1)} (S)M_i\right)\geq M-\frac{n(n-1)}{2},
    \end{align*}
    which is \eqref{eq:bornewronsk 1} after rearranging. Assuming furthermore that $\M{q} (\Theta_P^{(i-1)}(S_j))\geq 0 $ for all $i,j\in \llbracket 1,n \rrbracket$,
 every entry of the matrix has nonnegative multiplicity, so $\M{q} (M_i)  \geq 0$ which is better than \eqref{eq:min sur mi},
 and arguing as before we get

     \begin{align*}
        \M{q} (\mathcal{W}_P(x)) & = \M{q} \left(\sum_{i=1}^{n} \Theta_P^{(i-1)} (S)M_i\right)\geq\min_{1\leq i \leq n}( M-(i-1))\geq M-(n-1).
    \end{align*}
\end{proof}

\section{Case of a trinomial and a sparse polynomial}

In this section, we give a proof of the main \autoref{them:main}. We start by establishing \autoref{propsupportdedim1},
 \autoref{lemmebornval} and \autoref{lemme : deriv toric xaxb} as preliminary results.

\begin{proposition}\label{propsupportdedim1}
    Let $q\in (\Cc^*)^2$, $f$, $g \in \Cc[x,y]$, with $t_f$ and $t_g$ monomials respectively, and assume that the Newton polytope of $f$ 
    or of $g$ has dimension 1. 

If $\M{q}(f,g)<+\infty$ then
    $$
        \M{q} (f,g)\leq (t_f-1)(t_g-1).
    $$
\end{proposition}

\begin{proof}

Using the coordinate change \eqref{eq:cgmttore}, we can assume that $q=(1,1)$.
    By symmetry, we may assume that the Newton polytope of $f$ has dimension 1. 
    There exists a monomial change of coordinates $\Phi:\CC 2 \to \CC 2$ such that
    $$
        f\circ \Phi =f_0 
    $$
    where $f_0$ depends only on $x$. Since $f$ has $t_f$ monomials, the same is true for $f_0$ and we can write 
    $$
        f_0(x) = (1-x)^uf_1(x),
    $$
    where $f_1$ is a Laurent polynomial such that $f_1(1)\not=0$ and $0\leq u\leq t_f-1$. We get
    \begin{align*}
        \Mun (f,g) &= \Mun (f\circ \Phi,g\circ \Phi)\\
         &= \Mun ((1-x)^uf_1(x) ,g\circ \Phi(x,y))\\
         &= u\Mun (1-x,g\circ \Phi(x,y))\\
         &= u \M{1} (g \circ \Phi (1,y))\\
          & \leq (t_f-1)(t_g-1).
    \end{align*} 
    Indeed, $ \M{1} (g \circ \Phi (1,y))$ is finite because $\Mun (f,g)$ is finite, 
    and by Haj\'os' lemma, $ \M{1} (g \circ \Phi (1,y))\leq t_g-1$ since $g \circ \Phi (1,y)$ is a Laurent polynomial in one variable with at most $t_g$ monomials.
\end{proof}

\begin{lemme}\label{lemmebornval}
   Let $a,b\in\Cc^*$ such that $a+b=1$. Let $P=x(ax+b)$ and $g=g_1x^{a_{11}}y^{a_{21}}+\cdots + g_N x^{a_{1N}}y^{a_{2N}}$, 
   where $a_{ij}\in \Qq$, $g_j\in \Cc^*$ for $i\in \llbracket 1,2\rrbracket $ and $j \in  \llbracket 1,N\rrbracket $. Then
    $$
        \M{1}(g(x,ax+b))\leq N-1+\M{1}(W _P(g_1x^{a_{11}}(ax+b)^{a_{21}},\cdots , g_N x^{a_{1N}}(ax+b)^{a_{2N}})).
    $$
\end{lemme}

\begin{proof}

   We apply part \eqref{eq:bornewronsk 2} of \autoref{inegaliteWronskian} to 
    
    \begin{align*}
    \M{1} (g_1x^{a_{11}}(ax+b)^{a_{21}}+\cdots + g_N x^{a_{1N}}(ax+b)^{a_{2N}})  .  
    \end{align*}
    First, note that $P(1)=1\not=0$ and that, for any $i\in \llbracket 1,N \rrbracket$, the function $x^{a_{1i}}(ax+b)^{a_{2i}}$
     can be expanded as a power series in a neighborhood of 1.
Hence, for all $i,j\in \llbracket 1,N \rrbracket$,
     $$\M{1} (\Theta_P^{(j-1)}(x^{a_{1i}}(ax+b)^{a_{2i}}))\geq 0, $$
and the result follows from \autoref{inegaliteWronskian}.

\end{proof}

\begin{lemme}\label{lemme : deriv toric xaxb}
    Let $u,v\in \Qq$, $a,b\in\Cc$ and $P=x(ax+b)$. Then, for all $j\in \Nn$,
    \begin{align*}
        \Theta_P ^{(j)}(x^{u}(ax+b)^{v})=x^{u}(ax+b)^{v}Q_j
    \end{align*}
    where $Q_j\in \Cc[x]$ is of degree at most $j$.
\end{lemme}

\begin{proof}
    We proceed by induction on $j\in \Nn$. The case $j=0$, is trivial.

    Let $j \geq 1$ and assume that the statement holds for $j-1$. 
    We have
    \begin{align*}
        \Theta_P^{(j)} (x^{u}(ax+b)^{v})& =\Theta_P (x^{u}(ax+b)^{v}Q_{j-1})\\
        & = x^{u}(ax+b)^{v}(u(ax+b)Q_{j-1}+avxQ_{j-1}+x(ax+b)Q_{j-1}').
    \end{align*}
    Since $(u(ax+b)Q_{j-1}+avxQ_{j-1}+x(ax+b)Q_{j-1}')$ has degree at most $\deg (Q_{j-1})+1$, this concludes the proof.
\end{proof}

\begin{theorem}\label{them:main}\label{thm:associativity}
   Let $f\in \Cc[x,y]$ be a trinomial and $g\in \Cc[x,y]$ be a polynomial with $N\geq 2$ monomials.
   If $q\in \CC{2}$ is an isolated solution of $f=g=0$, then 
    \begin{align}\label{thrmgen}
        \M{q}(f,g) \leq \frac{(N-1)(N+2)}{2}.
    \end{align}
    Furthermore, for $N=2$ or $3$, this bound is sharp.
\end{theorem}

\begin{proof}

Consider the following system:
\begin{align*}
    f_0=g_0=0
\end{align*}
where $f_0\in \Cc [x,y]$ is a trinomial and $g_0\in \Cc [x,y]$ is a polynomial with $N\geq 2$ monomials. 
Up to a change of coordinates \eqref{eq:cgmttore}, which preserves the number of monomials, we may assume that $q=(1,1)$.
Since we work in the torus $\CC{2}$, we may divide $f_0$ by one of its monomials and assume that it has the form:
\begin{align*}
    f_0&=1+a_0x^{\alpha_1}y^{\alpha_2}+b_0x^{\beta_1}y^{\beta_2}.
\end{align*}
If the Newton polytope of $f$ has dimension 1, \autoref{propsupportdedim1} bounds
the multiplicity by $2(N-1) \leq \frac{(N-1)(N+2)}{2} $, and the result holds.
We may therefore assume that the Newton polytope of $f$ has dimension 2, i.e., $\alpha_1\beta_2-\beta_1\alpha_2\not=0$ and
$a_0,$ $b_0$ are non-zero.

Consider the monomial change of coordinates $\phi_A$ defined in \autoref{chgmmonomial} with 
$$
    A=\frac{1}{\alpha_1\beta_2-\beta_1\alpha_2}\begin{pmatrix}
        \beta_2 & -\beta_1 \\
        -\alpha_2 & \alpha_1
    \end{pmatrix}.
$$
Since $(1,1)\in \phi_A^{-1}(1,1)$, we have the following equality:
\begin{align*}
    \Mun (f_0,g_0) = \Mun (f_0\circ \phi_A , g_0\circ \phi _A).
\end{align*}
Then, up to rearranging the coefficients of $f$, we can consider the following simplified form of the system:
\begin{equation}\label{notations}
    \begin{split}
        f=f_0\circ \phi_A=&b+ax-y\\
        g=g_0\circ \phi _A=&g_1x^{a_{11}}y^{a_{21}}+\cdots + g_N x^{a_{1N}}y^{a_{2N}}
    \end{split}
\end{equation}
where $a_{ij}\in \Qq$, $g_j\in \Cc^*$ for $i\in \llbracket 1,2\rrbracket $, $j \in  \llbracket 1,N\rrbracket $ and
the exponent vectors $(a_{1j},a_{2j})$ are pairwise distinct.
Since $q=(1,1)$, we have
\begin{align*}
a+b=1.
\end{align*}
Applying the isomorphism
\begin{equation*}
    \Cc[[x-1, y-1]] / (f, g) \simeq \Cc[[x-1]] / (g(x, ax+b))
\end{equation*}
to the system \eqref{notations}, we obtain
\begin{equation}\label{eg}
    \Mun(f, g) = \M{1}\big(g(x, ax+b)\big).
\end{equation}
Then, by \autoref{lemmebornval}, this yields the bound
\begin{equation}\label{eq:inemult new}
    \Mun(f, g) \le N - 1 + \M{1}\big(\mathcal{W}(x)\big),
\end{equation}
where
\begin{equation*}
    \mathcal{W}(x) = W_P\big(g_1 x^{a_{11}}(ax+b)^{a_{21}}, \dots, g_N x^{a_{1N}}(ax+b)^{a_{2N}}\big).
\end{equation*}
We prove by induction on $N\geq 1$ that
  \begin{align*}
        \Mun(f,g) \leq \frac{(N-1)(N+2)}{2}.
    \end{align*}
If $N=1$, then $g$ is a monomial and does not vanish on $\CC{2}$, so 
\begin{align*}
    \Mun(f,g)=0\leq \frac{(1-1)(1+2)}{2}.
\end{align*}
Let $N\geq 2$ and assume the result for systems with fewer than $N$ monomials in $g$.
We will distinguish two cases, according to whether $\mathcal{W}$ vanishes identically.

Assume $\mathcal{W}(x) \equiv 0$. Then \autoref{prop : dep leineaire et Wronskian} shows that the functions $g_ix^{a_{1i}}(ax+b)^{a_{2i}}$ are linearly dependent.
 Then there exist complex numbers $\lambda_1, \dots , \lambda_N$ non all zero such that
\begin{align*}
    \sum_{i=1}^N \lambda_ig_ix^{a_{1i}}(ax+b)^{a_{2i}} =0.
\end{align*}
After renumbering and rescaling, we may assume $\lambda_N=-1$, so that
\begin{align*}
        g_Nx^{a_{1N}}(ax+b)^{a_{2N}} = \sum_{i=1}^{N-1} \lambda_ig_ix^{a_{1i}}(ax+b)^{a_{2i}} .
\end{align*}
Hence
$$
    g(x,ax+b)=g_1x^{a_{11}}(ax+b)^{a_{21}}+\cdots + g_{N-1} x^{a_{1N-1}}(ax+b)^{a_{2N-1}}+\sum_{i=1}^{N-1} \lambda_ig_ix^{a_{1i}}(ax+b)^{a_{2i}} 
$$
set
$$
\Tilde g = g_1x^{a_{11}}y^{a_{21}}+\cdots + g_{N-1} x^{a_{1N-1}}y^{a_{2N-1}}+\sum_{i=1}^{N-1} \lambda_ig_ix^{a_{1i}}y^{a_{2i}}.
$$
If $\tilde{g}(x,y) \equiv 0$ then $g(x,ax+b)\equiv 0$ and $\Mun (f,g)=\M{1} (g(x,ax+b))=+\infty $. Since $(1,1)$ is an isolated solution of $f=g=0$, then
 $\Mun (f,g)<+\infty $ and thus $\tilde{g}(x,y) \not \equiv 0$.

We obtain
\begin{align*}
    \Mun(f,g) = \Mun (g(x,ax+b))=\Mun(\Tilde g(x,ax+b))=\Mun(f,\Tilde g) .
\end{align*}
Since $\tilde{g} $ has at most $N-1$ monomials and $\tilde{g} \not \equiv 0$, we can use the induction hypothesis
\begin{align*}
     \Mun(f,g) = \Mun(f,\Tilde g)\leq  \frac{(N-2)(N+1)}{2}\leq \frac{(N-1)(N+2)}{2}.
\end{align*}
Assume now that $\mathcal{W}(x) \not\equiv 0$. Then by \autoref{lemme : deriv toric xaxb}

\begin{align*}
       \mathcal{W}(x)& =\sum_{\sigma\in  \mathfrak{S}_N }\epsilon(\sigma)\prod_{j=1}^{N}\Theta_P^{(j-1)}(g_{\sigma(j)}x^{a_{1\sigma(j)}}(ax+b)^{a_{2\sigma(j)}})\\
       &= \left(\prod_{j=1}^{N}g_{j}\right)x^{\sum_{i}a_{1i}}(ax+b)^{\sum_{i}a_{2i}}\sum_{\sigma\in  \mathfrak{S}_N }\epsilon(\sigma)\prod_{j=1}^{N}Q_{\sigma(j),j-1}
\end{align*}
Now, $\left(\prod_{j=1}^{N}g_{j}\right)x^{\sum_{i}a_{1i}}(ax+b)^{\sum_{i}a_{2i}}$ does not vanish at $x=1$,
 so $ \M{1}(\mathcal{W})$ equal the multiplicity of the polynomial $\sum_{\sigma\in \mathfrak{S}_N }\epsilon(\sigma)\prod_{j=1}^{N}Q_{\sigma(j),j-1}$, 
 which has degree at most
\begin{align}\label{deg}
    \sum_{j=1}^{N}(j-1) = \frac{N(N-1)}{2}.
\end{align}
Since $\mathcal{W} \not\equiv 0$, this polynomial is non zero, so its multiplicity at 1 is at most its degree; using \eqref{eq:inemult new}, we obtain
\begin{align}
     \Mun (f,g) = \M{1}(g(x,ax+b))&\leq N-1+\M{1} (\mathcal{W}(x)) \label{eqin}\\
     & \leq N-1+\frac{N(N-1)}{2} \\ 
     & =  \frac{(N-1)(N+2)}{2}
\end{align}
To establish the sharpness of the bound for $N=2$, it suffices to observe that
\begin{align*}
    \Mun ((x-1)^2,y-1)=2.
\end{align*}
To establish the sharpness of the bound for $N=3$, consider the following system
\begin{align}
    \begin{split}\label{syst de mult 5}
    f(x,y) &= 4-3x^6-y^{18}, \\
    g(x,y) &= x^{15}-4y^3+3xy^{10}.
    \end{split} 
\end{align}
To show that the multiplicity of this system at $(1,1)$ is equal to 5, we can use \eqref{eg}
 and compute the Taylor expansion around $x=1$ to the fifth order of $g\big(x,(4-3x^6)^{\frac{1}{18}}\big)$:
\begin{align*}
   g\big(x,(4-3x^6)^{\frac{1}{18}}\big) &= x^{15}-4(4-3x^6)^{\frac{1}{6}}+3x(4-3x^6)^{\frac{5}{9}} \\
   &= 24192(x-1)^5+O\big((x-1)^6\big).
\end{align*}

A way to find the system \eqref{syst de mult 5} can be found in the \autoref{appendix} to the present paper.

\end{proof}

\nocite{*}
\bibliographystyle{alpha}
\bibliography{Intersection_multiplicity_of_a_sparse_curve_and_a_trinomial_curve}

\appendix 
\addcontentsline{toc}{section}{APPENDICES}

\section{Appendix. System of two trinomials with multiplicity 5}\label{appendix}

To construct a system $f=g=0$ consisting of two trinomials in two variables with mutliplicity $5$ at $(1,1)$, we use the proof of \autoref{them:main}.
By \eqref{eqin} we should have $\M {1}(\mathcal{W}(x))=3$, and thus $1$ is a triple root of the polynomial
$C_3(x)=\sum_{\sigma\in \mathfrak{S}_3 }\epsilon(\sigma)\prod_{j=1}^{3}Q_{\sigma(j),j-1}$. By \eqref{deg}, this polynomial has degree $\leq 3$, thus it is equal
to $(x-1)^3$ up to a multiplicative constant.

Writing $g$ as
\begin{align*}
g(x,y)=g_1x^{a_{11}}y^{a_{21}}+g_2x^{a_{12}}y^{a_{22}}+g_3x^{a_{13}}y^{a_{23}}
\end{align*}
and dividing by $g_3x^{a_{13}}y^{a_{23}}$, we can assume that $g$ is equal to
\begin{align*}
g(x,y)=g_1x^{a_{11}}y^{a_{21}}+g_2x^{a_{12}}y^{a_{22}}+1.
\end{align*}

We know that
\begin{align*}
\mathcal{W}(x)=\W _P (g_1x^{a_{11}}(ax+b)^{a_{21}},g_2x^{a_{12}}(ax+b)^{a_{22}},1)
\end{align*}
is equal to $C_3(x) \cdot x^{a_{11}+a_{12}}(ax+b)^{a_{21}+a_{22}}$.
To ease the computation of $C_3(x)$, a Python algorithm is implemented, using SymPy. We get
\begin{align*}
C_3(x) = A_3(x-1)^3+A_2(x-1)^2+A_1(x-1)+A_0,
\end{align*}
where $A_0,A_1,A_2,A_3$ are explicit polynomials in $a_{11},a_{21},a_{12},a_{22},a,b$. 
To find a rational numbers $a_{11},a_{21},a_{12},a_{22}$ such that $A_0=A_1=A_2=0$ and $A_3\not=0$, we fix
$a=-3$ (and $b=4$).
Computing a Gröbner basis in lexicographic order (lex) of the ideal generated by the polynomials $A_0,A_1,A_2$, we find that the polynomial 
\begin{align*}
G(a_{11}, a_{21})=a_{11}^4 - 12a_{11}^3 a_{21} + 54 a_{11}^2 a_{21}^2 - a_{11}^2 - 108a_{11}a_{21}^3 + 174a_{11}a_{21} + 81a_{21}^4 - 81a_{21}^2
\end{align*}
(which depends solely on $a_{11}, a_{21}$) belongs to the ideal $(A_0,A_1,A_2)$. Using a basic search algorithm, we find many rational zeros of $G$. 
Taking one of them,
$$
\begin{matrix}
(a_{11},a_{21}) = (-\frac{5}{2}, \frac{1}{6}).
\end{matrix}
$$
By substituting $(a_{11},a_{21})$ with $(-\frac{5}{2}, \frac{1}{6})$ in the polynomials $A_0,A_1,A_2$,
it's not difficult to find rational values for $a_{12}$ and $a_{22}$, for instance
$$
\begin{matrix}
a_{12}=-\frac{7}{3}, & a_{22} = \frac{5}{9},
\end{matrix}
$$
such that $A_0=A_1=A_2=0$ and $A_3\not = 0 $.
We determine the coefficients $g_1$ and $g_2$ by requiring that $f$ and $g$ share a common tangent, ensuring a multiplicity of at least 2. 
This yields the following equality:
\begin{align*}
a=-\frac{g_1 a_{11}+g_2 a_{12}}{g_1 a_{21}+g_2 a_{22}},
\end{align*}
Furthermore, $g_1+g_2+1=0$; hence, $g_1=-4$ and $g_2=3$.

Using the passage to integer exponents via $(x,y)=(X^{6},Y^{18})$ and by multiplying $g$ by $x^{15}$,
we find the following explicit example of a system with multiplicity $5$ at $(1,1)$:
\begin{align}
f(x,y) &= 4-3x^6-y^{18}, \\
g(x,y) &= x^{15} - 4 y^3 + 3xy^{10}. 
\end{align}
This property is readily verified by substituting $y = (4-3x^6)^{\frac{1}{18}}$ into $g$ and computing its Taylor expansion around $x=1$ to the fifth order:
\begin{align*}
g\big(x,(4-3x^6)^{\frac{1}{18}}\big) &= x^{15}-4(4-3x^6)^{\frac{1}{6}}+3x(4-3x^6)^{\frac{5}{9}} \\
&= 24192(x-1)^5+O\big((x-1)^6\big).
\end{align*}

\end{document}